\documentclass[10pt,a4paper]{amsart}

\usepackage{amssymb}
\usepackage[final]{showkeys}
\usepackage{microtype}
\usepackage{color}
\usepackage{graphicx}
\usepackage[hmargin=3cm,vmargin={5cm,4cm}]{geometry}

\newtheorem{te}{Theorem}[section]
\newtheorem{example}{Example}

\newtheorem{os}[te]{Remark}
\newtheorem{prop}[te]{Proposition}

\numberwithin{equation}{section}

\allowdisplaybreaks

\begin{document}
	
	\title[]{A non-homogeneous generalization of Burgers equations.  }
\author{Francesco Maltese}
\maketitle

\begin{abstract}
  In this article we study generalizations of the inhomogeneous Burgers equation. First at the operator level, in the sense that we replace classical differential derivations by operators with certain properties, and then we increase the spatial dimensions of the Burgers equation, which is usually studied in one spatial dimension. This allows us, in one dimension, to find mathematical relationships between solutions of hyperbolic Brownian motion and the Burgers equations, which usually study the behaviour of mechanical fluids, and also, through appropriate transformations, to obtain in some cases exact solutions that depend on Hermite polynomials composed of appropriate functions.  In the multi-dimensional case, this generalization allows us, by means of the method of invariant spaces, to find exact solutions on Riemannian and pseudo-Riemannian varieties, such as Schwarzschild and Ricci Solitons space, with time dictated by fractional derivatives, such as a Caputo-type operator of fractional evolution.
	
	\bigskip
	
\end{abstract}
\maketitle

\section{Introduction}

In this paper we consider a generalization of the non-homogeneous Burgers type equation 

\begin{equation}
	\frac{\partial u}{\partial t}-A(t)\frac{\partial u^{2}}{\partial x}=\frac{\partial^{2} u}{\partial x^{2}}+A(t)u \hspace{0.2 cm} with  \hspace{0.2 cm} u=u(x,t)
\end{equation}

as previously addressed in [6], in an operator-like manner, that is, instead of the usual partial derivative operators that appear in equation (1.1) we substitute for $\frac{\partial }{\partial t}$ a generic $\mathcal{O}_t$ linear differential operator acting on the time variable $t$ and instead of the double derivative $\frac{\partial^{2}}{\partial x^{2}} $ we put two generic differential operators $N_x$ and $M_x$ linear, acting only on $x$-variable, combined in this way $N_{x}M_{x} $.

And finally, we want to substitute instead of the contribution $\frac{\partial u^{2}}{\partial x} $ the expression $M_{x}(uL(u))$ or  $N_{x}(uL(u))$ where $L$ is a linear operator acting only on $x$-variable ,  so that (1.1) becomes 

\begin{equation}
(a)\hspace{0.1cm}	\mathcal{O}_{t}u-A(t)M_{x}(uL(u))=N_{x}M_{x}u+A(t)u \hspace{0.2cm} or \hspace{0.2cm} (b) \hspace{0.1cm}\mathcal{O}_{t}u-A(t)N_{x}(uL(u))=N_{x}M_{x}u+A(t)u.
\end{equation} 

Equation (1.2) when the above operators are varied represents a class of equations where (1.1) becomes a special case of it if $M_{x}=N_{x}=\frac{\partial }{\partial x}$ and $\mathcal{O}_t=\frac{\partial }{\partial t}$ and $L$ is the identity operator.
\quad
Later in the following sections we will add appropriate assumptions to the operators just described in such a way as to find mathematical methods for solving equation (1.2); one such method, inspired by Vitanov's methods in [1], will consist of finding solutions to (1.2) by an appropriate transformation of the solutions to the linear equations (2.1) and this is covered in section 2.The other method, covered in section 3,  makes use of the invariant spaces used by Gazizov and Kasaktin in [2], which consists in transforming non-linear differential equations into linear ones by linear combinations of functions that remain invariant under the action of these equations(see appendix in section 5 ). Then again, in the last part of section 3, we will combine the two methods used in sections 2 and 3 to determine other solutions of (2.1) via the solutions (1.2) obtained by the invariant space method. At each solution method then we will find explicit examples of equations of the type (1.2) with admissible solutions.
In the last section we will make a further generalization about spatial dimensions by transforming (1.2) into (4.1). This spatial generalization allows us to introduce in (4.1) the Laplace-Beltrami operator, applied in hyperbolic space from its general form in [7], and then to determine the solutions of (4.1) on Riemannian and pseudo-Riemannian varieties. Furthermore, thanks to the method of invariant spaces, we can introduce into (4.1) the generalised Caputo fractional derivative, which is also used in [7]. Furthermore, (4.1) can be applied to Riemannian varieties whose metrics change in time thanks to the Ricci-Hamilton flow equation (see [8]), in particular we will apply it to some Ricci Solitons (see (4.31)).

\section{On the Obtaining Solutions of a non-homogeneous operator equation of Burgers type by Means of the Solutions of linear Differential operatorial Equations }

\quad

Consider the linear differential operator equation 
\begin{equation}
		\mathcal{O}_{t}\psi=M_{x}N_{x}\psi \hspace{0.2 cm} with  \hspace{0.2 cm} \psi=\psi(x,t)
\end{equation}
 with $\mathcal{O}_{t}, N_{x}, M_{x} $ the differential operators described in section 1 . As in Vitanov's article (for details see [1]) we want to find transformations that map any solutions of (2.1) into the solutions of (1.2) and as anticipated before for this purpose we add the following assumptions about the differential operators involved in (2.1)
 
 \begin{equation}
 	\mathcal{O}_{t}, N_{x},  M_{x}\hspace{0.1 cm} enjoy \hspace{0.1 cm}the\hspace{0.1 cm} Leibniz \hspace{0.1 cm}property \hspace{0.1 cm} for\hspace{0.1 cm}  for,\hspace{0.1 cm} respectively\hspace{0.1 cm} the\hspace{0.1 cm} variable\hspace{0.1 cm} t\hspace{0.1 cm} and\hspace{0.1 cm} x\hspace{0.1 cm} 
 \end{equation}

 \begin{equation}
 	\left[\mathcal{O}_{t}, M_{x} \right]=\left[\mathcal{O}_{t}, N_{x} \right]=0
 \end{equation}

\begin{equation}
M_{x}=LN_{x} \hspace{0.1 cm} with \hspace{0.1 cm} L\hspace{0.1 cm}the \hspace{0.1 cm} linear\hspace{0.1 cm} operator\hspace{0.1 cm} involved\hspace{0.1 cm} in\hspace{0.1 cm} equation \hspace{0.1 cm} (1.2)
\end{equation}

\begin{equation}
	\mathcal{O}_{t}A=-A^{2} \hspace{0.1 cm} and\hspace{0.1 cm} A \neq 0\hspace{0.1 cm} for\hspace{0.1 cm} each\hspace{0.1 cm} t \neq 0,\hspace{0.1 cm} with\hspace{0.1 cm}A=A(t)\hspace{0.1 cm} the\hspace{0.1 cm} coefficient\hspace{0.1 cm} function\hspace{0.1 cm} involved \hspace{0.1 cm}in\hspace{0.1 cm} (1.2) 	
\end{equation}	
	
\quad	
As a matter of synthesis we will merge the assumptions (2.2), (2.3), (2.4), (2.5) just listed by denoting them by the symbol H(2.1).

\quad

We now proceed to find a transformation by proving the following assertion

\begin{prop}
If the assumptions H(2.1) hold, and if there exists a function $ g= g(x,t)$ invertible  such that, given a function $u=u(x,t)$ 

\begin{equation}
	\mathcal{O}_{t}g(A(t)(N_{x})^{-1} u)=g(A(t)(N_{x})^{-1} u) \mathcal{O}_{t}(A(t)(N_{x})^{-1} u)
\end{equation}

\begin{equation}
	 M_{x}g(A(t)(N_{x})^{-1} u)=g(A(t)(N_{x})^{-1} u) M_{x}(A(N_{x})^{-1} u )\hspace{0.1cm} and\hspace{0.1cm}  N_{x}g(A(t)(N_{x})^{-1} u)=g(A(t)(N_{x})^{-1} u) N_{x}(A(t)(N_{x})^{-1} u )
\end{equation}
 Then $u=\frac{1}{A(t)}N_{x} g^{-1}(\psi(x,t))$ is a solution of (1.2)(b) if a solution $\psi=\psi(x,t)$ of (2.1) can be written as $\psi(x,t)=g(A(t)(N_{x})^{-1} u)$.
\end{prop}
	\begin{proof}
		
		By hypothesis
		
		$$ \mathcal{O}_{t}\psi=M_{x}N_{x}\psi             $$
		
	becomes
	
	$$  \mathcal{O}_{t} g(A(t)(N_{x})^{-1} u) =M_{x}N_{x} g(A(N_{x})^{-1} u)                     $$
	
	$$  \psi \mathcal{O}_{t}(A(t)(N_{x})^{-1} u)= M_{x}(\psi N_{x}(A(t)(N_{x})^{-1} u ) )                                                $$
	
	$$  \psi \mathcal{O}_{t}(A(t)(N_{x})^{-1} u)=  \psi M_{x}N_{x} (A(t)(N_{x})^{-1} u)+ \psi M_{x}(A(t)(N_{x})^{-1} u ) N_{x}(A(t)(N_{x})^{-1} u ).                                           $$
\quad
Now we divide both members of the equation by $\psi$ so we get

$$\mathcal{O}_{t}(A(t)(N_{x})^{-1} u)=M_{x}N_{x} (A(t)(N_{x})^{-1} u)  +   M_{x}(A(t)(N_{x})^{-1} u ) N_{x}(A(t)(N_{x})^{-1} u )                                      $$

$$  A(t)\mathcal{O}_{t}((N_{x})^{-1} u)  + \mathcal{O}_{t}(A(t))(N_{x})^{-1}u=A(t)M_{x}N_{x}(N_{x})^{-1} u +A^{2}(t)M_{x}(N_{x})^{-1} u N_{x}(N_{x})^{-1} u                                            $$

\begin{equation}
 A(t)\mathcal{O}_{t}((N_{x})^{-1} u)  + \mathcal{O}_{t}(A(t))(N_{x})^{-1}u=A(t)M_{x} u +A^{2}(t)(M_{x}(N_{x})^{-1} u) u .                                  
\end{equation}
 \quad
 By (2.5) of the hypothesis H(2.1) we have that
 $\mathcal{O}_{t}A(t)=-A(t)^{2}$ so (2.8) becomes
 
 $$  A(t)\mathcal{O}_{t}((N_{x})^{-1} u)   -A^{2}(t)(N_{x})^{-1}u=A(t)M_{x} u +A^{2}(t)(M_{x}(N_{x})^{-1} u) u  $$

 and dividing both members of (2.8) by A(t), we get
  
  \begin{equation}
    \mathcal{O}_{t}((N_{x})^{-1} u)   -A(t)(N_{x})^{-1}u=M_{x} u +A(t)(M_{x}(N_{x})^{-1} u) u                                        
\end{equation}
  
furthermore using hypothesis (2.7), we have that $L=M_{x} (N_{x})^{-1}$ and (2.9) becomes

\begin{equation}
 \mathcal{O}_{t}((N_{x})^{-1} u)   -A(t)(N_{x})^{-1}u=M_{x} u +A(t)(L u) u .                           
\end{equation}
By deriving both members of (2.9) for the operator $N_{x}$ and taking advantage of assumption (2.3) of assumptions H(2. 1) i.e., $ \left[\mathcal{O}_{t}, N_{x} \right]=0$ (a nontrivial assumption since it might be the case that the time operator $\mathcal{O}_{t}$ might depend on the spatial $x$-coordinate), (2.10) becomes

$$  \mathcal{O}_{t}u   -A(t)u=N_{x}M_{x} u +A(t)N_{x}((L u) u ) .                                                                  $$

Which becomes (1.2)(b)

\begin{equation}
 \mathcal{O}_{t}u   - A(t)N_{x}( u(L u) )                                         =N_{x}M_{x} u +A u .                                                           
\end{equation}

The solution $u$ of (2.11) is derived from the relation 
$\psi(x,t)= g(A(t)(N_{x})^{-1} u)$ by solving it with respect to $u$ and thus the thesis is proved.

\end{proof}

\quad

Still maintaining the same assumptions of H(2.1) and of (2.6) and (2.7), we can start from equation (2.11) and proceed in the opposite direction to obtain the $\psi$ solution of equation (2.1) from $u=\frac{1}{A(t)}N_{x} g^{-1}(\psi)$ 

\begin{prop}
If the assumptions H(2.1) hold, and if there exists a function $ g= g(x,t)$ invertible  such that if 	$u=\frac{1}{A}N_{x} g^{-1}(\psi)$ , with $\psi=\psi(x,t)$, it is a solution of (2.11) for which we have that

\begin{equation}
	\mathcal{O}_{t}(\psi)=\psi \mathcal{O}_{t}(A(t)(N_{x})^{-1} u)
\end{equation}

\begin{equation}
	N_{x}\psi=\psi N_{x}(A(t)(N_{x})^{-1} u )\hspace{0.1cm} and\hspace{0.1cm}M_{x}\psi=\psi M_{x}(A(t)(N_{x})^{-1} u )
\end{equation}

Then $\psi$ is a solution of equation (2.1) in particular $\psi$ is written as $\psi(x,t)= g(A(t)(N_{x})^{-1} u)$

\end{prop}

\begin{proof}
	In the demonstration it is enough to substitute $u=\frac{1}{A(t)}N_{x} g^{-1}(\psi)$ into (2.11) and use the assumptions of the theorem . Moreover, the relation between $u$ and $\psi$ can be seen as a generalization of the Cole- Hopf-Type transform, so the proof can be seen as a generalization of the calculation to solve nonlinear nonhomogeneous Burgers equation (1.1) ( to see such calculation in a little more detail refer to page 2 of [6] )
\end{proof}

\subsection{Subfamily of solutions of Burgers-type operator equations for $L=\mathbb{Id}$ and generalized Hermite polynomials}

\quad

Let us begin with the case in which in equation (1.2) we have that in addition to the fact that $L=\mathbb{Id}$ the spatial differential operators are equal i.e. $N_{x}=M_{x}$ in which case (2.1) becomes

\begin{equation}
	 \mathcal{O}_{t}u   - A(t)N_{x} ( u^{2}    )                                      =N_{xx}u +A(t) u .                    
\end{equation}

Where by $N_{xx}$ we mean to use the operator $N_{x}$ twice i.e. $N_{xx}=N_{x}N_{x}$.

\quad

If the assumptions of proposition (2.1) hold true, then a possible solution of (2.14) is $u(x,t)=\frac{1}{A(t)}N_{x} g^{-1}(\psi(x,t))$ where $ \psi(x,t)$ for (2.1) is solution of the operator equation of the heat-type equation

\begin{equation}
	\mathcal{O}_{t}\psi=N_{xx}\psi
\end{equation}

with an initial condition for an appropriate value of the variable $t$. This value is fixed at $t_{1}$ such that $\psi(x,t_{1})=f(x)^{n}$ for some integer $n>1$ with $f(x)$ a fixed function.
If the function $f(x)$ verifies that $N_{x}f(x)=1$ and there exists a function $h(t)$ such that $\mathcal{O}_{t} h(t)=1$ and $h(t_{1})=0$. Then a possible solution of (2.15) is a generalized Hermite polynomial i.e. 

\begin{equation}
H_{n}(f(x),h(t))=n!\sum_{r=0}^{[\frac{n}{2}]}	\frac{h(t)^{r}f(x)^{n-2r}}{(n-2r)!r!}
\end{equation}
In that case then a possible solution of (2.14) will be

\begin{equation}
	u(x,t)=\frac{1}{A(t)}N_{x} g^{-1}(H_{n}(f(x),h(t)))
\end{equation}

where we recall that the function $g$ is an appropriate function that satisfies the assumptions of proposition (2.1).

\quad

We introduce two possible examples of (2.14) in the hyperbolic Poincairé half-plane in the polar coordinates of a hyperbolic surface. First, we will use the coordinate of the arc length of geodesics $\eta$ (for details on how to derive these coordinates and on the hyperbolic half-plane of Poincairé see [7])

\begin{example}
	
\quad
	
With following differential operators $N_{\eta}=\sinh \eta \partial_{\eta}$ and $ \mathcal{O}_{t}=\partial_{t}$  the equation (2.14) with such differential operators becomes
	
\begin{equation}
	\partial_{t}u-A(t)\sinh \eta \partial_{\eta}(u^{2})=sinh \eta \partial_{\eta}sinh \eta \partial_{\eta} u +A(t)u
\end{equation}	
	
 With following function $g(\xi(\eta,t))=e^{\xi(\eta,t)}$	the differential operators just described not only verify hypotheses H(2.1) but also verify hypotheses (2.6) and (2.7) in fact
 
 $$N_{\eta}g(\xi(\eta,t))=\sinh \eta \partial_{\eta}(e^{\xi(\eta,t)})$$
 
 $$=\sinh \eta e^{\xi(\eta,t)}\partial_{\eta}\xi(\eta,t)=e^{\xi(\eta,t)}\sinh \eta \partial_{\eta}\xi(\eta,t)=e^{\xi(\eta,t)}N_{\eta}\xi(\eta,t)$$
 
 $$ \mathcal{O}_{t}g(\xi(\eta,t))=\partial_{t}e^{\xi(\eta,t)}=e^{\xi(\eta,t)}\partial_{t}\xi(\eta,t)=e^{\xi(\eta,t)}\mathcal{O}_{t}\xi(\eta,t).$$
 
 So starting from (2.15), which in this context becomes
 
 \begin{equation}
 	\partial_{t}\phi=\sinh \eta \partial_{\eta}\sinh \eta \partial_{\eta} \phi
 \end{equation}
 
By means of the transformation $\psi(\eta,t)= g(A(t)(N_{\eta})^{-1} u(\eta,t))$ ,which becomes in this case $ \psi(\eta,t)=e^{A(t)\int_{\eta_{0}}^{\eta}\frac{u(s,t)}{\sinh s}\, ds }$, where $\eta_{0}$ is a suitably fixed coordinate, we will go by proposition (2. 1) to equation (2.18). So for what we have seen so far a possible solution of (2.18) is

\begin{equation}
	u(\eta,t)=\frac{1}{t-t_{0}}\sinh \eta\partial_{\eta}\ln(H_{n}(\ln \tanh \frac{\eta}{2}, t))=\frac{1}{t-t_{0}}n\frac{ H_{n-1}(\ln \tanh \frac{\eta}{2}, t)}{H_{n}(\ln \tanh \frac{\eta}{2}, t)}. 
\end{equation} 
	Where we have that $H_{n}(\ln \tanh \frac{\eta}{2}, t)$ is the generalized Hermite polynomial that solves equation (2.19) with the initial condition $\psi(\eta,0)=(\ln \tanh \frac{\eta}{2})^{n}$ for some fixed integer $n$. In fact, $N_{\eta}ln \tanh \frac{\eta}{2}=\sinh \eta \partial_{\eta}   \ln \tanh \frac{\eta}{2}=1$ and obviously $\mathcal{O}_{t} t=\partial_{t} t=1$ with $t_{1}=0$.
	Moreover, the function $A(t)=\frac{1}{t-t_{0}}$ solves condition (2.5) of the list of hypotheses H(2.1), since $\mathcal{O}_{t} A(t)=-A^{2}(t)$ becomes in that case $\partial_{t}A(t)=-A^{2}(t)$ 
	
\end{example}

\quad
\begin{example}
	In the second example, we introduce the following operators $$N_{\eta}=\frac{1}{\sinh \eta } \partial_{\eta}$$ 
	$$\mathcal{O}_{t}=\partial_{t}$$. Furthermore, as function $g$ of proposition (2.1) we choose the following function $g(\xi(\eta,t))=e^{\xi(\eta,t)}$. Then equation (2.14) in that case becomes.
	
	\begin{equation}
			\partial_{t}u-A(t)\frac{1}{\sinh \eta }  \partial_{\eta}(u^{2})=\frac{1}{\sinh \eta }\partial_{\eta}\frac{1}{\sinh \eta } \partial_{\eta} u +A(t)u.
	\end{equation}
Taking advantage of the assumptions of proposition (2.1) which are in that case satisfied, to solve (2.21) we will start from (2.15) which in that case becomes

\begin{equation}
\partial_{t}\psi=\frac{1}{\sinh \eta } \partial_{\eta}\frac{1}{\sinh \eta } \partial_{\eta} \psi	
\end{equation}
and also here, we make use of the transformation $\psi(\eta,t)= g(A(N_{\eta})^{-1} u(\eta,t))$ which in that case becomes $\psi(u)=e^{A\int_{\eta_{0}}^{\eta}u(s,t)\sinh s\, ds }$, where $\eta_{0}$ is a suitably fixed coordinate. 

\quad
Then a possible solution of (2.21) is obtained by the same method as in Example 1 by means of an appropriate Hermite polynomial for some positive integer $n$ which in this case will be $H_{n}(\cosh \eta, t)$, and it is of the type $$	u(\eta,t)=\frac{1}{t-t_{0}}\frac{1}{\sinh \eta} \partial_{\eta}\ln(H_{n}(\cosh\eta, t))=\frac{1}{t-t_{0}}n\frac{ H_{n-1}(\cosh \eta, t)}{H_{n}(\cosh \eta, t)} .$$

\end{example}

\subsection{Burgers-type operator equations in the case $L \neq \mathbb{I}$}

\quad

We introduce an example in the case $L \neq \mathbb{I}$ on the hyperbolic half-plane of Poincairé . Interestingly, this case relates the solution of (1.2) in this case to a well-known solution of Brownian motion on a hyperbolic half-plane ( for details refer to [3] (1.5), (1.6), (1.7), (2.14), (2.15)).

\begin{example}
	In the hyperbolic half-plane of Poincairé we introduce the following linear operator $L=(\frac{1}{\sinh \eta })^{2}\cdot \mathbb{I}$ and the three differential operators $N_{\eta}=\sinh \eta \partial_{\eta}$ and $M_{\eta}=\frac{1}{\sinh \eta }\partial_{\eta}$, $\mathcal{O}_{t}=\partial_{t}$ to transform (1.2) (b).

\begin{equation}
	\partial_{t}u-A(t)\sinh \eta \partial_{\eta}\Bigl[\Bigl(\frac{1}{\sinh \eta }\Bigr)^{2} u^{2} \Bigr]=\sinh \eta \partial_{\eta}\frac{1}{\sinh \eta }\partial_{\eta}u(t) + A(t)u.
\end{equation}

Introducing as an invertible function the function used in the other two examples, namely $g(\xi(\eta, t))=e^{\xi(\eta,t)}$, we verify that with the operators illustrated above we have the assumptions of proposition (2.1). So we can find a possible solution for (2.23) via equation (2.1), which, , when subjected to the application of such operators, becomes

\begin{equation}
\partial_{t}\phi=\frac{1}{\sinh \eta }\partial_{\eta}\sinh \eta \partial_{\eta} \phi \hspace{0.2 cm}i.e. \hspace{0.2 cm} \partial_{t}\phi= \Delta_{H}\phi
\end{equation}

Where $\Delta_{H}$ is the hyperbolic Laplacian  and (2.24) is the differential equation of hyperbolic Brownian motion (for further details see [3] and also [5]).

\quad

By means of the transformation $ \phi(\eta,t)=e^{A(t)\int_{\eta_{0}}^{\eta}\frac{u(s,t)}{\sinh s}\, ds }$ (2.24) goes into (2. 23); therefore the solution of (2.23) will be of the type $ u(\eta,t)=\frac{1}{A(t)}\sinh \eta \partial_{\eta}ln(\phi(\eta,t))$ and since the solution of (2.24) is

\begin{equation}
	\phi(\eta,t)=\frac{e^{-(t/4)}}{\sqrt{\pi}(\sqrt{2t})^{3}}\int_{\eta}^{+ \infty}\frac{\psi e^{-(\psi^{2}/4t)}}{\sqrt{\cosh \psi-\cosh \eta}}\,d\psi.	
\end{equation}

\quad

then a possible solution, unless initial conditions are met, of (2.23) becomes
\begin{equation}
	 u(\eta,t)=\frac{1}{t-t_{0}}\sinh \eta\partial_{\eta}ln\Biggr(\frac{e^{-(t/4)}}{\sqrt{\pi}(\sqrt{2t})^{3}}\int_{\eta}^{+ \infty}\frac{\psi e^{-(\psi^{2}/4t)}}{\sqrt{\cosh \psi-\cosh \eta}}\,d\psi\Biggl).
\end{equation}

However, it should be emphasized to choose an appropriate domain of $\phi$ because in (2.26) it is an argument of the natural logarithm function $\ln(\cdot)$. Also as can be seen in [3] the initial condition of (2.24) is Dirac's $delta$ for $t=0$ i.e. $\phi(\eta,0)=\delta(\eta)$ and this quite problematic in (2.26) again due to the fact of the natural logarithm function. Therefore, it will be appropriate to choose a $t \neq 0$ that exhibits good behaviour in the initial condition of (2.24).
\end{example}

\quad

\section{Solving the Burgers-type operator equation via invariant spaces}

In this section we want to solve  (1.2) using the invariant subspaces method used in [2](see section 5) and justified theoretically by R.K .Gazizov and A.A Kasaktin. Nevertheless, we must recognize that our analysis is based on certain assumptions.
\quad

Suppose we have a $\mathbb{R}$-vector space of functions $W=<\omega_{l}(x)>_{l=0}^{s}$ for which the following properties hold

\begin{equation}
N_{x}(\omega_{l}(x))=M_{x}(\omega_{l}(x))=1 \hspace{0.2 cm}\forall l\neq 0 \hspace{0.2cm} and \hspace{0.2cm} N_{x}(\omega_{0}(x))=M_{x}(\omega_{0}(x))=0\hspace{0.1 cm} or\hspace{0.1 cm}M_{x}(\omega_{l}(x))=1 \hspace{0.2 cm}\forall l\neq 0 \hspace{0.2cm} and \hspace{0.2cm} M_{x}(\omega_{0}(x))=0
\end{equation}

\begin{equation}
	W \le Spec L(\lambda(t))\hspace{0.2 cm} where \hspace{0.2 cm}\lambda(t)\hspace{0.2 cm} is\hspace{0.2 cm} a\hspace{0.2 cm} function\hspace{0.2 cm} dependent\hspace{0.2 cm} on\hspace{0.2 cm} t
\end{equation}

\quad
In assumption (3.2) we are meaning that $W$ is a subspace of the spectrum of the linear operator $L$ of (1.2), of eigenvalue a certain function $\lambda(t)$ i.e. $W$ is contained in the eigenspace associated with the eigenvalue $\lambda(t)$ of $L$. 

Furthermore, if we add the following property 

\begin{equation}
N_{x},  M_{x}\hspace{0.1 cm} enjoy \hspace{0.1 cm}the\hspace{0.1 cm} Leibniz \hspace{0.1 cm}property \hspace{0.1 cm} with \hspace{0.1 cm} respect\hspace{0.1 cm} to\hspace{0.1 cm} the\hspace{0.1 cm} variable \hspace{0.1 cm}x.
\end{equation}

Note that (3.1) and (3.3) imply that $W \leq Ker N_{x}M_{x} $ this will be useful in computing the solution (1.2)(a) and (b) by the method of invariant spaces.

\quad

Due to the assumptions (3.1) , (3.2) and (3.3) we can look for solutions of (1.2) in the following functions

\begin{equation}
	u(x,t)=\sum_{l=0}^{s}b_{l}(t)\omega_{l}(x),
\end{equation}

where $b_{l}(t)$ are unknown functions in the variable $t$ to be determined for (3.4) to satisfy (1.2)(a) or (b) and this is the objective of the present study.

\quad
Recall what we have just derived from the hypotheses just assumed that $W \leq Ker N_{x}M_{x} $ . Furthermore, by (3.1) we have that $M_{x}(u)= \sum_{j=1}^{s}b_{j}(t)$ or $M_{x}(u)=N_{x}(u)=\sum_{j=1}^{s}b_{j}(t)$ and the equation (2.1) in both cases (a) and (b) can be expressed as follows, in accordance with (3.2) and (3.3):

$$ \mathcal{O}_{t}(u)-2A(t)\lambda(t)(\sum_{j=1}^{s}b_{j}(t))u= A(t)u .$$

\quad

Now to determine the differential system satisfying the coefficients $b_{l}(t)$ we write the (1.2) just simplified in terms of the solution (3.4) and determine the new coefficients of $\omega(x)$ that equaled $0$ will give rise to the new equations

$$\mathcal{O}_{t}(u)=A(t)(2\lambda(t) \sum_{j=1}^{s}b_{j}(t)+1)u $$
$$\mathcal{O}_{t}(\sum_{l=0}^{s}b_{l}(t)\omega_{l}(x))= A(t)(2\lambda(t) \sum_{j=1}^{s}b_{j}(t)+1)\sum_{l=0}^{s}b_{l}(t)\omega_{l}(x)  $$         

\begin{equation}                                                  
\sum_{l=0}^{s} \mathcal{O}_{t}(b_{l}(t))\omega_{l}(x)= \sum_{l=0}^{s}b_{l}(t)A(t)(2\lambda(t)\sum_{j=1}^{s}b_{j}(t)+1)  \omega_{l}(x) .                                               
\end{equation}

\quad

We can then establish that (3.4) is the solution of (1.2) by determining the functions $b_{l}(t)$ from (3.5), combining the assumptions (3.1), (3.2), (3.3) in the list called H(3.1) .

\begin{prop}
		If H(3.1) assumptions hold, then the function $$ u(x,t)=\sum_{l=0}^{s}b_{l}(t)\omega_{l}(x) $$ is a possible solution of (1.2) if $b_{l}(t)$ are solutions of the following differential system
		
		\begin{equation}
			\mathcal{O}_{t}(b_{l}(t))=A(t)b_{l}(t)[2\lambda(t)\sum_{j=1}^{s}b_{j}(t)+1]\hspace{0.2cm} \forall l \in \{0,...,j,...,s \}.
		\end{equation}
\end{prop}

\quad

Note that the assumptions H(3.1) do not involve the operator $\mathcal{O}_{t}$. Thus, the only assumption is that $\mathcal{O}_{t}$ is linear and differential with respect to functions of the type $f(t)$. There is an advantage in this, in that we have a chance to choose differential operators other than linear combinations of partial derivatives with respect to $t$. In fact, we can solve (3.6) by choosing as $\mathcal{O}_{t}$ a Caputo-type operator of fractional evolution $\mathcal{O}_{\beta, f}^{t}$ with $\beta>0 \in (0,1)$ a real number and $f(t)$ a increasing function of the type $L^{1}[0,t]$ fixed with $f(0)=0$ and ${ f}^{'}(t)\neq 0$ . That is

\begin{equation}
	(\mathcal{O}_{\beta, f}^{t}b)(t)=\frac{1}{\Gamma(n-\beta)}\int_{0}^{t}f^{'}(\tau)(f(t)-f(\tau))^{n-\beta-1}\Bigr(\frac{1}{f^{'}(\tau)}\frac{d}{d\tau}\Bigl)^{n}b(\tau)\,d\tau ,
\end{equation}

with $n=[\beta]+1$ (to see the more general definition of the operator $\mathcal{O}_{\beta, f}^{t}$ and its use, see the texts [5] and [7] given in the bibliography ).

\quad

Therefore, equation (3.6) becomes.

\begin{equation}
		\mathcal{O}_{\beta, f}^{t}(b_{l}(t))=A(t)b_{l}(t)[1+2\lambda(t)\sum_{j=1}^{s}b_{j}(t)].
\end{equation}

\quad

In order to simplify the situation and be able to solve (3.8) we could assume that 

\begin{equation}
	s=1, \hspace{0.1cm} A(t)=\frac{C}{2\lambda(t)E_{\beta,1}(C(f(t))^{\beta})+1},\hspace{0.1cm} A(t)[1+2\lambda(t)b_{1}(t)]=C\hspace{0.1cm} with \hspace{0.1cm} C\hspace{0.1cm} a \hspace{0.1cm} real\hspace{0.1cm} constant
\end{equation}

\quad

where $ E_{\beta,1}(C(f(t))^{\beta}) $ is a Mittag-Leffler composite function ( for a more detailed examination of this function, please refer to [3]).

By virtue of (3.9) the system (3.8)  becomes 

\begin{equation}
\begin{cases}	
\mathcal{O}_{\beta,	f}^{t}(b_{1}(t))=Cb_{1}(t) \\
\mathcal{O}_{\beta,	f}^{t}(b_{0}(t))=Cb_{0}(t).
\end{cases}	
\end{equation}

And so  $b_{1}(t)$, $ b_{0}(t)$ are an eigenfunctions of the operator $\mathcal{O}_{\beta, f}^{t}$ so we will have that 

\begin{equation}
	b_{1}(t)=E_{\beta,1}(C(f(t))^{\beta}),\hspace{0.1 cm} b_{0}(t)=E_{\beta,1}(C(f(t))^{\beta}).
\end{equation}

\quad

 Let us now look at some examples in which to apply prop (3.1)

 \quad

 \begin{example}
 	Suppose that $N_{x}=M_{x}=\partial_{x}$ and that $\mathcal{O}_{t}=\mathcal{O}_{\beta, f}$ and furthermore $L=\mathbb{I}\cdot(\cdot)$  then with such operators (1.2) becomes 
 	
 	\begin{equation}
 		\mathcal{O}_{\beta,	f}^{t}u-\frac{C}{2E_{\beta,1}(C(f(t))^{\beta})+1}\partial_{x}(u^{2})=\partial_{xx}u+\frac{C}{2E_{\beta,1}(C(f(t))^{\beta})+1}u.
 	\end{equation}

 	The objective is to determine whether the problem can be solved using proposition (3.1). First, it is easy to see that there exists a vector space generated by the polynomial functions $1,x$ i.e. $W=<1, x>$ that satisfies together with the operators with respect to the spatial coordinate the assumptions H(3.1). It follows therefore that, in the absence of suitable initial conditions, a potential solution can be described as follows $$ u(x,t)=b_{1}(t)x+ b_{0}(t).$$
 		
 		\quad
 		
 		To completely determine the solutions just illustrated being $A(t)=\frac{C}{2\lambda(t)E_{\beta,1}(C(f(t))^{\beta})+1} $ and $A(t)[1+2\lambda(t)b_{1}(t)]=C$  with $C$  a real constant  and $s=1$ then we are in condition (3.9). Then the coefficients $b_{l}(t)$ are of the type (3.11),and so we have a possible  solutions for (3.12).
 		
 		\begin{equation}
 			u(x,t)= E_{\beta,1}(C(f(t))^{\beta})(x+1).
 		\end{equation}
 	
 \end{example}

\quad
Now to conclude the examples of proposition (3.1) instead of the usual Euclidean space in the $x$ coordinate we move to the hyperbolic Poincairé half-plane in the $\eta$ coordinate.

\begin{example}
	Now in (1.2) we consider the following operators in the hyperbolic poincairé half-plane  $N_{\eta}=M_{\eta}=\sinh \eta\partial_{\eta}$,$\mathcal{O}_{t}=\mathcal{O}_{\beta, f}$ with $L=\mathbb{I}\cdot(\cdot)$ 
	and as an additional assumption we assume that $A(t)=\frac{C}{2E_{\beta,1}(C(f(t))^{\beta})+1}$  then (1.2) becomes
	\begin{equation}
		\mathcal{O}_{\beta,	f}^{t}u-\frac{C}{2E_{\beta,1}(C(f(t))^{\beta})+1}\sinh \eta\partial_{\eta}(u^{2})=\sinh \eta\partial_{\eta}\sinh \eta\partial_{\eta}u+\frac{C}{2E_{\beta,1}(C(f(t))^{\beta})+1}u.
	\end{equation}
\quad

Proceeding as Example 4 the reference invariant space for solutions are $W=<\ln(\tanh \frac{\eta}{2 }),1>$ again we will determine the coefficients $b_{l }(t)$ as in the previous example. So the possible solution are of the type

\begin{equation}
		u(\eta,t)= E_{\beta,1}(C(f(t))^{\beta})\Bigr(\ln \Bigr(\tanh \frac{\eta}{2 }\Bigl)+1\Bigl).
\end{equation}

\end{example}
\quad

On this last example it is worth making an observation by slightly modifying (3.14).

\begin{os}

Starting from the last example with the same conditions except that $N_{\eta}=\frac{1}{\sinh \eta}\partial_{\eta}$ and we get (1.2)(a) with the hyperbolic Laplacian $\Delta_{H}$ i.e.

\begin{equation}
		\mathcal{O}_{\beta,	f}^{t}u-\frac{C}{2E_{\beta,1}(C(f(t))^{\beta})+1}\sinh \eta\partial_{\eta}(u^{2})=\Delta_{H}u+\frac{C}{2E_{\beta,1}(C(f(t))^{\beta})+1}u.
\end{equation}
It is verified that by the method of Example 5 the solutions are the same i.e. they are (3.15) despite the fact that we have changed the operator $N_{\eta}$. Moreover, the linear operator $L$, which in that case is an identity, does not create any relationship between the operators $N_{\eta}$ and $M_{\eta}$ but to apply proposition 3.1 such an assumption is not necessary and this is advantageous since it gives us more possibilities to solve such equations .
\end{os}

\subsection{Solving the operator equation of heat type by the solutions of the nonhomogeneous operator equation of Burgers type by invariant spaces}

\quad

Let us try to apply proposition (2.2) to solve the linear operator equation (2.1) in which case we must have the set of assumptions H(2.1) . These assumptions involve the operator $\mathcal{O}_{t}$ to have more stringent constraints to before and in that case we cannot make use of fractional differential operators.

\quad

We then begin to obtain some examples the solutions of (1.2) via invariant spaces and then apply prop 2.2 to obtain the solutions of (2.1) . Then we need the hypotheses H(2.1) and H(3.1).

\quad

\begin{example}
		We take PDE (3.12) again but instead of the operator $\mathcal{O}_{\beta, f}^{t}$ we put the classic partial derivative operator i.e. $\mathcal{O}_{t}=\partial_{t}$ to make the hypothesis H(2.1) hold. So (3.12) becomes that at the beginning of the article the (1.1) i.e.
		
		$$
				\partial_{t}u-A(t)\partial_{x}(u^{2})=\partial_{xx}u+A(t)u.
		$$
		
		\quad
	We solve (1.1) by invariant subspaces method; obviously it is the same as in Example 4 i.e. $W=<x, 1>$ . We now calculate the coefficients $b_{l}$ with $l=0,1$ of the solution for prop 3.1 such coefficients satisfy the following generally nonlinear ODE system
	
	\begin{equation}
	\begin{cases}
		\partial_{t}b_{1}(t)=A(t)b_{1}(t)(1+2b_{1}(t))\\
			\partial_{t}b_{0}(t)=A(t)b_{0}(t)(1+2b_{1}(t))\\
	\end{cases}	
	\end{equation}
	\quad
	We try to simplify the situation by looking for the particular solution of the type $b_{1}(t)=b_{0}(t)=b(t)$ then (3.17) reduces to one equation
	
	\begin{equation}
		\partial_{t}b(t)=\frac{1}{(t-t_{0})}(b(t)+ 2b^{2}(t)).
	\end{equation}
\quad
Here it is implied that $A(t)=\frac{1}{(t-t_{0})}$ and this is due to the hypothesis H(2.1) remembering that $t_{0}$ is an oppportune initial condition.

\quad
Equation (3.18) is nonlinear but can be solved by quadratures through separation of variables; in fact, after a few algebraic steps we arrive at integrating the following equation $$ \int_{b(t_{1})}^{b(t)}\frac{1}{y+2y^{2}}\,dy=\int_{t_{1}}^{t}\frac{1}{(s-t_{0})}\,ds       .$$ 
Where $t_{1}$ is an appropriate initial condition that it is convenient that $t_{1}\neq t_{0}$ to avoid any criticality.
Integrating the first member integral by the usual standard method for integrating rational algebraic functions, we obtain the following solution

\begin{equation}
	b(t)=\frac{c(t-t_{0})}{1-2c(t-t_{0})}.
\end{equation}
	Where $c$ is an appropriate real number determined by the initial conditions just set and the integration operations.
	So the solution of (1.1) can be
	
	\begin{equation}
		u(x,t)=\frac{c(t-t_{0})}{1-2c(t-t_{0})}(x+1).
	\end{equation}
	
	\quad
	
	The function $g(\xi(x, t))=e^{\xi(x,t)}$
	already used in the section 2 is the right choice to solve equation (2.1), which in that case we recall is $$ \partial_{t} \psi= \partial _{xx}\psi .$$ Since it satisfies (2.12) and (2.13) the function $e^{\xi(x,t)}$ allows us to apply proposition 2.2. So in addition to the Hermite polynomials as solutions of (2.1) we get other solutions of the type
	
  $$\psi(x,t)=g(A(t)(N_{x})^{-1}u(x,t))=g\Bigl(\frac{c(t-t_{0})}{1-2c(t-t_{0})}\int_{x_{0}}^{x}(s+1)\,ds\Bigr)=e^{\frac{c(t-t_{0})}{1-2c(t-t_{0})}\int_{x_{0}}^{x}(s+1)\,ds}                $$
  
  \begin{equation}
  	\psi(x,t)=e^{\frac{c(t-t_{0})}{1-2c(t-t_{0})}(\frac{x^{2}}{2}+x+c_{1})}. 
  \end{equation}
\quad
Where we always remember that $x_{0}$ is an appropriate initial condition chosen for the spatial part of $u(x,t)$ and $c_{1}$ is a constant integral derived from the choice of $x_{0}$.

\end{example}

\begin{example}
	
	In this example, we will proceed in a similar manner to that employed in Example 6, thus we will omit a few steps.
	First of all, as we did before, we place $\mathcal{O}_{t}=\partial_{t}$ and then substitute it to $\mathcal{O}_{\beta, f}^{t}$ at (3.14).
	Then (3.14) becomes equation (2.18) of Example 1. The invariant space of (2.18) is the same as (3. 14) i.e., $W=<\ln(\tanh \frac{\eta}{2 }),1>$, while the coefficients $b_{1}(t), b_{0}(t)$ satisfy the system (3.17) but again we consider a particular solution in which $b_{1}(t)=b_{0}(t)=b(t)$ and it is (3.19). So in addition to (2.20) a solution of (2.18) is as follows
	
	\begin{equation}
		u(x,t)=\frac{c(t-t_{0})}{1-2c(t-t_{0})}\Bigl(\ln\Bigl(\tanh \frac{\eta}{2 }\Bigr)+1\Bigr)
	\end{equation}

\quad

In order to solve differential equation (2.22), which we recall is $$ \partial_{t}\phi=\sinh \eta \partial_{\eta} \sinh \eta \partial_{\eta}  \phi ,$$ we will make use of the function $g(\xi(\eta , t))=e^{\xi(\eta,t)}$ that verifies the hypotheses (2.12) and (2.13) and allow us to find the solution by thanks of Proposition 2.2.
Thus we can then reload other kinds of solutions for (2.22) besides the generalized Hermite polynomials (2. 20).One such solution is given by the following: $$\psi(\eta,t)=g(A(t)(N_{x}^{-1}u(\eta,t))=g\Bigl(\frac{c(t-t_{0})}{1-2c(t-t_{0})}\int_{\eta_{0}}^{\eta}\frac{1}{\sinh s}\Bigl(\ln\Bigl(\tanh \frac{s}{2 }\Bigr)+1 \Bigr)\, ds\Bigr)$$
$$\psi(\eta,t)=e^{\frac{c(t-t_{0})}{1-2c(t-t_{0})}\int_{\eta_{0}}^{\eta}\frac{1}{\sinh s}(\ln(\tanh \frac{s}{2 })+1) \,ds} $$

\begin{equation}
\psi(\eta,t)=e^{\frac{c(t-t_{0})}{1-2c(t-t_{0})}\Bigl(\frac{(\ln(\tanh \frac{\eta}{2 })+1)^{2}}{2}+ c_{1}\Bigr) }  .                             
\end{equation}

\quad
Again $\eta_{0}$ is an appropriate fixed initial condition and $c_{1}$ is a constant derived from that condition  through integration .

\end{example}

\quad
And now we begin with the last section of the article

\section{A dimensional generalization of the non-homogeneous Burgers-type operator equation}

\quad
We want to introduce the following dimensional generalization of equation PDE (1.2)

\begin{equation}
(a)	\mathcal{O}_{t}u-\sum_{d=1}^{m}A_{d}(t){M(d)}_{x_{d}}(uL_{d}(u))=(\sum_{d=1}^{m}{N(d)}_{x_{d}}{M(d)}_{x_{d}})u+(\sum_{d=1}^{m}A_{d}(t))u \hspace{0.1 cm}
\end{equation}
$$(b) \mathcal{O}_{t}u-\sum_{d=1}^{m}A_{d}(t){N(d)}_{x_{d}}(uL_{d}(u))=(\sum_{d=1}^{m}{N(d)}_{x_{d}}{M(d)}_{x_{d}})u+(\sum_{d=1}^{m}A_{d}(t))u.$$

\quad
Where the function $u=u(x_{1},...,x_{j},...,x_{m},t)$ is a scalar function dependent on $m$ spatial variables. While $N(d)_{x_{d}}$ and $M(d)_{x_{d}}$can be defined as a sequence of linear differential operators with respect to the spatial variable $x_{d}$. And $L_{d}$ is a linear operator with respect to functions of the type $f(x_{1},...,x_{j},..., x_{m})$. And all these characteristics just listed hold as the index $d \in \{ 1,...,j,...,m \}$ changes.

\quad

We attempt to find a solution to (4.1) by adding the assumption (3.3) for the operators $N(d)_{x_{d}}$ and $M(d)_{x_{d}}$ for each $d$. Then we could try to find an admissible solution of the type $u=\sum_{d=1}^{m}u_{d}(x_{d},t)$ where in turn the contributions of this summation are of the type $u_{d}(x_{d},t)=\sum_{l_{d}=0}^{m_{d}}b_{d}^{(l_{d})}(t)\omega_{d}^{(l_{d})}(x_{d})$ with $m_{d}>0$ for each $d$. On the functions $\omega_{d}^{(l_{d})}(x_{d})$ we assume that the assumptions (3.1) hold with respect to the operators $N(d)_{x_{d}}$ and $M(d)_{x_{d}}$  and in addition we add the fact that 

\begin{equation}
	L_{d}(\omega_{d^{'}}^{(l_{d^{'}})}(x_{d^{'}}))=\lambda_{d}(t)(\omega_{d^{'}}^{(l_{d^{'}})}(x_{d^{'}}))\hspace{0.2cm} \forall d^{'} \in \{ 1,...,j,...,m \},
\end{equation}

  always for each $d$, where with $\lambda_{d}(t)$ is the associated eigenvalue.  In summary, we can name this list of hypotheses, which on the dimensional index $d$ thus $H(3.1)_{d}$.
\quad

It is easy to see that under the assumptions $H(3.1)_{d}$ for each $d$ the following properties occur

\begin{equation}
	N(d)_{x_{d}}(u)=M(d)_{x_{d}}(u)=\sum_{l_{d}=1}^{m_{d}}b_{d}^{(l_{d})}(t) \hspace{0.2cm} \hspace{0.1cm}in\hspace{0.1cm}(4.1)(b)\hspace{0.1cm} or \hspace{0.2cm} M(d)_{x_{d}}(u)=\sum_{l_{d}=1}^{m_{d}}b_{d}^{(l_{d})}(t)\hspace{0.1cm}in\hspace{0.1cm}(4.1)(a)
\end{equation}
	
\begin{equation}
	N(d)_{x_{d}}M(d)_{x_{d}}u=0
\end{equation}	

\begin{equation}
	L_{d}(u)=\lambda_{d}(t)u.
\end{equation}

\quad

In this context, then, finding an admissible solution for (4.1) consists in finding a system of differential equations compatible with (4.1) for the coefficients $b_{d}^{(l_{d})}(t)$ as $l_{d}$ and $d$ vary. And this is done by the method of invariant spaces. Let us see how.

\quad

First, by the distributive property and by (4.5) and (4.4), (4.1) in case (a) becomes

$$	\mathcal{O}_{t}u-\sum_{d=1}^{m}A_{d}(t)\lambda_{d}(t)M(d)_{x_{d}}(u^{2})=\sum_{d=1}^{m}A_{d}(t)u , $$

and in case (b) becomes

$$ \mathcal{O}_{t}u-\sum_{d=1}^{m}A_{d}(t)\lambda_{d}(t)N(d)_{x_{d}}(u^{2})=\sum_{d=1}^{m}A_{d}(t)u . $$

And having assumed that $u=\sum_{d=1}^{m}u_{d}$, we have that in (4.1)(a)

$$\sum_{d=1}^{m}\mathcal{O}_{t}u_{d}-2\sum_{d=1}^{m}A_{d}(t)\lambda_{d}(t)[M(d)_{x_{d}}(u_{d}) (\sum_{d^{'}=1}^{m}u_{d^{'}})    ]=\sum_{d=1}^{m}A_{d}(t)(\sum_{d^{'}=1}^{m}u_{d^{'}}  )
	 $$
	 
and in (4.1)(b) we have that

$$\sum_{d=1}^{m}\mathcal{O}_{t}u_{d}-2\sum_{d=1}^{m}A_{d}(t)\lambda_{d}(t)[N(d)_{x_{d}}(u_{d}) (\sum_{d^{'}=1}^{m}u_{d^{'}})    ]=\sum_{d=1}^{m}A_{d}(t)(\sum_{d^{'}=1}^{m}u_{d^{'}}  )
.$$	 
	 
\quad

Here the linearity of all differential operators was exploited and the square of the solution $u$ was developed by means of the Leibniz property enjoyed by the operators $N(d)_{x_{d}}$ and $M(d)_{x_{d}}$ for the hypotheses $H(3.1)_{d}$, in addition to the distributive property of the last summation at the second member.

\quad

 The equation can be simplified. In fact, we have that

$$\sum_{d=1}^{m}\mathcal{O}_{t}u_{d}=2\sum_{d=1}^{m}A_{d}(t)\lambda_{d}(t)[M(d)_{x_{d}}(u_{d}) (\sum_{d^{'}=1}^{m}u_{d^{'}})    ]+\sum_{d=1}^{m}A_{d}(t)(\sum_{d^{'}=1}^{m}u_{d^{'}}  ) \hspace{0.1 cm } in \hspace{0.1cm} case \hspace{0.1cm} (4.1)(a) $$

$$\sum_{d=1}^{m}\mathcal{O}_{t}u_{d}=2\sum_{d=1}^{m}A_{d}(t)\lambda_{d}(t)[N(d)_{x_{d}}(u_{d}) (\sum_{d^{'}=1}^{m}u_{d^{'}})    ]+\sum_{d=1}^{m}A_{d}(t)(\sum_{d^{'}=1}^{m}u_{d^{'}}  ) \hspace{0.1 cm } in \hspace{0.1cm} case \hspace{0.1cm} (4.1)(b). $$

And in either case, the result will be the same.

$$ \sum_{d=1}^{m}\mathcal{O}_{t}u_{d}=(\sum_{d^{'}=1}^{m}u_{d^{'}})[\sum_{d=1}^{m}2A_{d}(t)\lambda_{d}(t)(\sum_{l_{d}=1}^{m}b_{d}^{(l_{d})}(t))+\sum_{d=1}^{m}A_{d}(t) ]  .$$

$$ \sum_{d=1}^{m}\mathcal{O}_{t}u_{d}=(\sum_{d=1}^{m}u_{d})[\sum_{d^{'}=1}^{m}2A_{d^{'}}(t)\lambda_{d^{'}}(t)(\sum_{l_{d^{'}}=1}^{m}b_{d^{'}}^{(l_{d^{'}})}(t))+\sum_{d^{'}=1}^{m}A_{d^{'}}(t) ]     .$$

\quad
Now we can substitute for the functions $u_{d}$ the expressions $\sum_{l_{d}=0}^{m_{d}}b_{d}^{(l_{d})}(t)\omega_{d}^{(l_{d})}(x_{d})$ and the equations then become

\begin{equation}
 \sum_{d=1}^{m} \sum_{l_{d}=0}^{m_{d}}\mathcal{O}_{t}(b_{d}^{(l_{d})}(t))\omega_{d}^{(l_{d})}(x_{d})=  \sum_{d=1}^{m} \sum_{l_{d}=0}^{m_{d}} b_{d}^{(l_{d})}(t)[\sum_{d^{'}=1}^{m}2A_{d^{'}}(t)\lambda_{d^{'}}(t)(\sum_{l_{d^{'}}=1}^{m}b_{d^{'}}^{(l_{d^{'}})}(t))+\sum_{d^{'}=1}^{m}A_{d^{'}}(t) ]\omega_{d}^{(l_{d})}(x_{d}).                      
\end{equation}	 
\quad
Equalizing to $0$ the new coefficients of the functions $\omega_{d}^{(l_{d})}(x_{d})$ formed in equation (4.6) we obtain the following propertys

\begin{prop}
	Equation (4.1) under the assumptions $H(3.1)_{d}$ can have as possible solution of the type
	
	\begin{equation}
		u(x_{1},...,x_{j},...,x_{m},t)=\sum_{d=1}^{m}\sum_{l_{d}=0}^{m_{d}}b_{d}^{(l_{d})}(t)\omega_{d}^{(l_{d})}(x_{d}).
	\end{equation}
If the coefficients $b_{d}^{(l_{d})}(t)$ satisfy the following system of differential equations

\begin{equation}
\mathcal{O}_{t}(b_{d}^{(l_{d})}(t))=b_{d}^{(l_{d})}(t)[\sum_{d^{'}=1}^{m}2A_{d^{'}}(t)\lambda_{d^{'}}(t)(\sum_{l_{d^{'}}=1}^{m}b_{d^{'}}^{(l_{d^{'}})}(t))+\sum_{d^{'}=1}^{m}A_{d^{'}}(t) ]	 \hspace{0.1cm} \forall\hspace{0.1cm}  d\hspace{0.1cm}  \in \{ 1,...,j,...,m\}\hspace{0.1cm} and\hspace{0.1cm} \forall\hspace{0.1cm} l_{d} \in \{ 0,...,j_{d},...,m_{d}\}.
\end{equation}

\end{prop}

\quad

In general, the system of equations (4.8) is not at all linear and it can be complicated to find exact solutions. Therefore, we can add constraint on this equation to make it simpler and find exact solutions. For example, if in (4.8)  we have as operator $\mathcal{O}_{t}=\mathcal{O}_{\beta, f}^{t}$ the system (4.8) becomes

\begin{equation}
\mathcal{O}_{\beta, f}^{t}(b_{d}^{(l_{d})}(t))	=b_{d}^{(l_{d})}(t)[\sum_{d^{'}=1}^{m}2A_{d^{'}}(t)\lambda_{d^{'}}(t)(\sum_{l_{d^{'}}=1}^{m}b_{d^{'}}^{(l_{d^{'}})}(t))+\sum_{d^{'}=1}^{m}A_{d^{'}}(t) ]	 \hspace{0.1cm} \forall\hspace{0.1cm}  d\hspace{0.1cm}  \in \{ 1,...,j,...,m\}\hspace{0.1cm} and\hspace{0.1cm} \forall\hspace{0.1cm} l_{d} \in \{ 0,...,j_{d},...,m_{d}\}
\end{equation}
                
and we add the following constraint

\begin{equation}
	\sum_{d^{'}=1}^{m}2A_{d^{'}}(t)\lambda_{d^{'}}(t)(\sum_{l_{d^{'}}=1}^{m}b_{d^{'}}^{(l_{d^{'}})}(t))+	\sum_{d^{'}=1}^{m}A_{d^{'}}(t)=A \hspace{0.1cm} with \hspace{0.1cm} A \hspace{0.1cm} a \hspace{0.1cm} real\hspace{0.1cm} constant .
\end{equation}
\quad
This strategy of simplifying the system can be seen as  the generalization in the one-dimensional case to a single spatial coordinate (see (3.9)). Indeed, as in the case $d=1$, case, the choice of $b_{d}^{(l_{d})}(t)$ must satisfy the system (4.10) and (4.11), and this is probably not always possible.

Should it be, the system (4.10) becomes linear

\begin{equation}
	\mathcal{O}_{\beta, f}^{t}(b_{d}^{(l_{d})}(t))=b_{d}^{(l_{d})}(t)A  \hspace{0.1cm} \forall\hspace{0.1cm}  d\hspace{0.1cm}  \in \{ 1,...,j,...,m\}\hspace{0.1cm} and\hspace{0.1cm} \forall\hspace{0.1cm} l_{d} \in \{ 0,...,j_{d},...,m_{d}\}.
\end{equation}

\quad
And so (4.7) becomes

\begin{equation}
		u(x_{1},...,x_{j},...,x_{m},t)=E_{\beta,1}(A(f(t))^{\beta})\sum_{d=1}^{m}\sum_{l_{d}=0}^{m_{d}}\omega_{d}^{(l_{d})}(x_{d}).
\end{equation}

\quad

Let us now see how to apply what has been said in the following example.

\begin{example}
	We are in the hyperbolic half-plane and will consider the hyperbolic Laplacian in its entirety in the hyperbolic coordinates i.e.
	
	$$ \Delta_{H}=\frac{1}{\sinh \eta}\partial_{\eta}\sinh \eta \partial_{\eta} + \frac{1}{{\sinh}^{2}\eta} \partial_{\alpha \alpha}         $$ 
	 For further discussion, see [7].
	 \quad
	 In this context, we consider the following differential and non-differential operators
	 
	 $$  M(1)_{\eta}=\sinh \eta \partial_{\eta} , \hspace{0.2 cm} N(1)_{\eta}=\frac{1}{\sinh \eta}\partial_{\eta}$$
	 
	 $$  M(2)_{\alpha}= \partial_{\alpha} , \hspace{0.2 cm} N(2)_{\alpha}=\frac{1}{{\sinh}^{2}\eta }\partial_{\alpha}                                                           $$
	 
	 $$  L_{\eta}=L_{\alpha}=\mathbb{I}, \hspace{0.2cm} \mathcal{O}_{t}=\mathcal{O}_{\beta, f}^{t}        .  $$
\quad
Then (4.1)(a) becomes

\begin{equation}
	\mathcal{O}_{\beta, f}^{t}u- A_{\eta}(t)\sinh \eta \partial_{\eta} (u^{2})-A_{\alpha}(t)\partial_{\alpha} (u^{2})=\Delta_{H}u+(A_{\eta}(t)+A_{\alpha}(t))u
\end{equation}	 

\quad
with $u=u(\eta, \alpha, t)$.

\quad

 The space of invariant functions satisfying the hypotheses $H(3.1)_{\eta}$ and $H(3.1)_{\alpha}$ consists of the following generators
 
 $$  \omega_{\eta}^{(1)}(\eta)= \ln \tanh \frac{\eta}{2}   , \hspace{0.2cm}  \omega_{\eta}^{(0)}(\eta) =1                                   $$
 $$ \omega_{\alpha}^{(1)}(\alpha)= \alpha   , \hspace{0.2cm}  \omega_{\alpha}^{(0)}(\alpha) =1   .                                                        $$

\quad

To obtain a solution like (4.12) we need to find $A_{\eta}(t)$, $A_{\alpha}(t)$ compatible with constraint (4.10) and system (4.11) for an appropriate choice of a real constant $A$. In this case the constraint (4.10) becomes

\begin{equation}
	2A_{\eta}(t)b_{\eta}^{(1)}(t) +2A_{\alpha}(t)b_{\alpha}^{(1)}(t)+A_{\eta}(t)+A_{\alpha}(t)=A. 
\end{equation}

\quad
Among the many solutions that can solve the constraint (4.14) compatible with the system (4.11) we can choose the following

\begin{equation}
	A_{\eta}(t)=A_{\alpha}(t)=\frac{A}{2[1+2E_{\beta,1}(A(f(t))^{\beta})]} .
\end{equation}

\quad

So to recapitulate, if among the coefficients $b_{\eta}^{(1)}(t)$, $b_{\alpha}^{(1)}(t)$ we choose those that depend on $ A_{\eta}(t)$, $ A_{\alpha}(t)$ according to the relation (4.15), we will obtain compatibility with (4.14) and (4.11) and thus the possible solution will be

\begin{equation}
		u(\eta,\alpha,t)=E_{\beta,1}(A(f(t))^{\beta})(\ln \tanh \frac{\eta}{2}+ \alpha+ 2).
\end{equation}
\end{example}

\subsection{A dimensional generalization of the nonhomogeneous burgers type operator equation on curved space-times}
\quad

In this section we will use differential varieties applied to equation (4.1), where in the operator $\sum_{d=1}^{m}N(d)_{x_{d}}M(d)_{x_{d}}$ we will make the Laplace-Beltrami operator induced by the metric of the pseudo-Riemannian variety and Riemannian varieties

\begin{equation}	
\Delta=\frac{1}{\sqrt{|g|}}\sum_{i=1}^{n}\partial_{i}\big(\sum_{j=1}^{n}\sqrt{|g|}g^{ij}\partial _{j}\big).
\end{equation}
\quad
Where $n$ is the dimension of the variety, $g^{ij}$ is the element of the inverse matrix of the metric tensor $ g_{ij}$ and $|g|$ is the  modulus  of the determinant of the metric tensor matrix; for further study of the topics Riemannian and pseudo-Riemannian varieties with related metrics and the Laplace-Beltrami operator, one can consult the text [4] in the bibliography. 
 However, (4.17) has already been used in the Riemannian variety of the Poincairè semiplane in Example 8; this is precisely the hyperbolic Laplacian $\Delta_{H}$, and it is from a generalization of (4.13) in terms of metric tensors that we will develop this topic. Let us see how.
 
 \quad
 
Given a metric tensor $g_{ij}$ associated with a variety $M$ of dimension $n$. We make use of (4.1) in these terms; by imposing

\begin{equation}
 \Delta_{M}:=\sum_{d=1}^{m}N(d)_{x_{d}}M(d)_{x_{d}}=\sum_{d=1}^{n}\frac{1}{\sqrt{|g|}}\partial_{d}\big(\sum_{j=1}^{n}\sqrt{|g|}g^{dj}\partial _{j}\big)
\end{equation}

\begin{equation}	
		The \hspace{0.1cm}operators\hspace{0.1cm} M(d)_{x_{d}}\hspace{0.1cm} and\hspace{0.1cm} N(d))_{x_{d}} are\hspace{0.1cm} used\hspace{0.1cm} in\hspace{0.1cm} equations\hspace{0.1cm} (4.1)(a)\hspace{0.1cm} \forall \hspace{0.1cm} d\hspace{0.1cm} \in \{1,2,...,m \} .
\end{equation}
\quad
 The differential operators $N(d)_{x_{d}}$,$M(d)_{x_{d}}$ will be chosen appropriately both to satisfy the assumptions $H(3.1)_{d}$ and also to obtain exact solutions. We will use metrics that will allow this.
Howewer, in this way we will obtain the following PDE.

\begin{equation}
\mathcal{O}_{t}u-\sum_{d=1}^{n}A_{d}(t)	M(d)_{x_{d}}[uL_{d}(u)]=\sum_{d=1}^{n}\frac{1}{\sqrt{|g|}}\partial_{d}\big(\sum_{j=1}^{n}\sqrt{|g|}g^{dj}\partial _{j}u\big)+(\sum_{d=1}^{m}A_{d}(t))u.
\end{equation}

Note that a diagonal metric would greatly simplify equation (4.20), and we will make use of such metrics.

\begin{example}
	In this example we will solve the equation (4.20)  in a gravitational field  generated by a $M$ mass which induces a Schwartzschild metric, that we recall is
	
	\begin{equation}
		ds^{2}=\Big(1-\frac{2GM}{c^{2}r}\Big)c{^2}dt^{2}-\Big(1-\frac{2GM}{c^{2}r}\Big)^{-1}dr^{2}-r^{2}d\theta^{2}- r^{2}\sin^{2}\theta d\phi^{2}.
	\end{equation}

\quad
Where $G$ is the universal gravitational constant and $c$ is the speed of light and $M$ is the source mass generating the gravitational field . The environment is the curved spacetime of coordinates $(t,r,\theta,\phi)$, where the spatial part $(r,\theta,\phi)$ are polar coordinates. We will not go into purely physical matters for example near the singularities of the metric tensor where these are static black holes and the event horizon is a spherical surface and without charge. For further discussion see ch 10[4].
\quad
Also in (4.20) we will treat $t$ as a spatial coordinate, which means that we will have another time variable if we want to consider it that way, or a parameter called $\tau$. So the solution function to (4.20) will be of the type $u=u(\tau,t,r,\theta, \phi)$.

\quad
Let us now calculate $|g|$ and hereafter $\sqrt{|g|}$

$$ g_{\mu \nu}=\begin{vmatrix}
(1-\frac{2GM}{c^{2}r})c{^2}& 0& 0&0 \\ 0&-(1-\frac{2GM}{c^{2}r})^{-1}&0&0\\ 0&0& -r^{2}&0 \\ 0&0&0& - r^{2}\sin^{2}\theta	
\end{vmatrix} $$
So
 $$ g =det(g_{\mu \nu})=-c{^2}\Big(1-\frac{2GM}{c^{2}r}\Big )\Big( 1-\frac{2GM}{c^{2}r}\Big )^{-1}r^{4}\sin^{2}\theta=-c{^2}r^{4}\sin^{2}\theta $$
 
 $$\sqrt{|g|}=\sqrt{c^{2}r^{4}\sin^{2}\theta}=cr^{2}\sin\theta.$$
 
 If we denote by $\Delta_{S}$ the Laplace-Beltrami operator associated with the Schwarzschild metric, we can compute each individual contribution of $\Delta_{S}$ which depends only on the diagonal elements of the inverse matrix
 
 $$g^{\mu \nu}=\begin{vmatrix}
 	\frac{1}{(1-\frac{2GM}{c^{2}r})^c{^2}}& 0& 0&0 \\ 0&-(1-\frac{2GM}{c^{2}r})&0&0\\ 0&0&-\frac{1}{r^{2}}&0 \\ 0&0&0& -\frac{1}{r^{2}\sin^{2}\theta	} 
 \end{vmatrix} $$

 , which we denote as $(\Delta_{S})_{d}$ where $d \in \{ t, r, \theta, \phi \}$.
After some more or less laborious calculations we get

$$(\Delta_{S})_{t}=\frac{r}{c^{2}r-2GM}\partial_{tt} \hspace{0.5cm}(\Delta_{S})_{r}=\frac{1}{(cr)^{2}}\partial_{r}[(2GM-c^{2}r)r] \partial_{r}$$

$$(\Delta_{S})_{\theta}=-\frac{1}{r^{2}\sin \theta}\partial_{\theta}\sin \theta\partial_{\theta} \hspace{0.5cm}(\Delta_{S})_{\phi}=-\frac{1}{(r\sin \theta)^{2}}\partial_{\phi \phi} .
$$ 
Given that $\Delta_{S}=(\Delta_{S})_{t}+(\Delta_{S})_{r}+(\Delta_{S})_{\theta}+(\Delta_{S})_{\phi}$ we have that

\begin{equation}
	\Delta_{S}=\frac{r}{c^{2}r-2GM}\partial_{tt}+\frac{1}{(cr)^{2}}\partial_{r}[(2GM-c^{2}r)r] \partial_{r}-\frac{1}{r^{2}\sin \theta}\partial_{\theta}\sin \theta\partial_{\theta}-\frac{1}{(r\sin \theta)^{2}}\partial_{\phi \phi} .
\end{equation}
\quad
Furthermore, knowing that $N(d)_{x_{d}}M(d)_{x_{d}}=(\Delta_{S})_{d}$ we can choose the following operators

\begin{equation}
M(t)_{t}=\partial_{t}, \hspace{0.2cm}M(r)_{r}=(2GM-c^{2}r)r \partial_{r}, \hspace{0.2cm}M(\theta)_{\theta}=\sin \theta \partial_{\theta}, \hspace{0.2cm} M(\phi)_{\phi}=\partial_{\phi}.
\end{equation}

\begin{equation}
N(t)_{t}=\frac{r}{c^{2}r-2GM}\partial_{t}, \hspace{0.2cm}N(r)_{r}=\frac{1}{(cr)^{2}}\partial_{r},\hspace{0.2cm}N(\theta)_{\theta}=-\frac{1}{r^{2}\sin \theta}\partial_{\theta}, \hspace{0.2cm}N(\phi)_{\phi}=-\frac{1}{(r\sin \theta)^{2}}\partial_{\phi}.  
\end{equation}
 And finally by posing $\mathcal{O}_{\tau}=\mathcal{O}_{\beta, f}^{\tau}$ and $L_{d}=\mathbb{I} \hspace{0.2cm}\forall d \in \{t, r, \theta, \phi \}$ the equation (4.20) becomes
 
 \begin{equation}
 	\mathcal{O}_{\beta, f}^{\tau}u-A_{t}(\tau)\partial_{t}(u^{2})-A_{r}(\tau)(2GM-c^{2}r)r \partial_{r}(u^{2}) -A_{\theta}(\tau)\sin \theta \partial_{\theta}(u^{2})-A_{\phi}(\tau)\partial_{\phi}(u^{2})=\end{equation}
 $$\Delta_{S}u+(A_{t}(\tau)+A_{r}(\tau)+A_{\theta}(\tau)+A_{\phi}(\tau))u.$$
 
 \quad
 After some calculations of standard one-variable integration we obtain the invariant functions $ \omega_{d}^{i}(x_{d})$ for equation (4.25).
 
 \begin{equation}
 	\omega_{t}^{(1)}(t)=t, \hspace{0.2cm} \omega_{t}^{(0)}(t)=1,\hspace{0.2cm}	\omega_{r}^{(1)}(r)=\frac{1}{2GM}\ln\Big(\frac{r}{2GM-c^{2}r}\Big) ,\hspace{0.2cm}\omega_{r}^{(0)}(r)=1
 \end{equation}

\begin{equation}
	\omega_{\theta}^{(1)}(\theta)=\ln \tan \frac{\theta}{2}, \hspace{0.2cm} \omega_{\theta}^{(0)}(\theta)=1,\hspace{0.2cm} 	\omega_{\phi}^{(1)}(\phi)=\phi,\hspace{0.2cm} \omega_{\phi}^{(0)}(\phi)=1.
\end{equation}
\quad
 There are the premises to apply Proposition (4.1) to find the solution. We will use the same method as in Example 8 to obtain the coefficients $b_{d}^{(l_{d})}(\tau)$. In fact, by placing the coefficients $A_{d}(\tau)$ in such a way
 
 \begin{equation}
 	A_{t}(\tau)=A_{r}(\tau)=A_{\theta}(\tau)=A_{\phi}(\tau)=\frac{A}{4[1+2E_{\beta,1}(A(f(\tau))^{\beta})]}. 
 \end{equation} 
 
 For a real number $A$.
 \quad
 
 For (4.28) a possible solution for (4.25) is
 
 \begin{equation}
 	u(\tau, t, r, \theta, \phi)=E_{\beta,1}(A(f(\tau))^{\beta})\Big(t+\frac{1}{2GM}\ln\Big(\frac{r}{2GM-c^{2}r}\Big)+\ln \tan \frac{\theta}{2}+\phi+4\Big).
 \end{equation}

\end{example}
\quad

In the latter examples it has been implied but still occasionally it should be remembered that the solutions contain logarithmic functions, so if it is appropriate we will restrict the domain of the solution on the varieties where equations (4.13) and (4.25) are defined.
\quad
Let's move on to the final example, which concludes this article.

\begin{example}
	This example starts with a metric obtained from the Ricci-Hamilton equation
	
	\begin{equation}
		\partial_{t}g_{ij}=-2R_{ij}.
	\end{equation}

\quad
\quad

Where a Riemannian metric $g_{ij}$ evolves over time proportionally to the Ricci tensor $R_{ij}$.  In particular, we consider those special metric solutions of (4.30) called $Ricci \hspace{0.2cm} Solitons$, which have the characteristic that they can change their size but their shape remains the same unless diffeomorphisms. Among these solutions we choose this one in particular

\begin{equation}
	ds^{2}=\frac{dx^{2}+dy^{2}}{e^{4t}+x^{2}+y^{2}}.
\end{equation}
Called $Cigar \hspace{0.2cm} Soliton$ (see ch1 of [8] and (1.1.18) for details).
\quad
We skip a few steps, since they are similar to Example 9 to obtain an equation of the type (4.20). Except that the metric depends on time which is the same as equation (4.20), which in this case becomes
\begin{equation}
	\mathcal{O}_{\beta, f}^{t}u-A_{x}(t)\partial_{x}u^{2}-A_{y}(t)\partial_{y}u^{2}=(e^{4t}+x^{2}+y^{2})(\partial_{xx}+\partial_{yy})u+(A_{x}(t)+A_{y}(t))u.
\end{equation}
\end{example}

Where $u=u(t,x,y)$. As before for the following choice

\begin{equation}
A_{x}(t)=A_{y}(t)=\frac{A}{2[1+2E_{\beta,1}(A(f(\tau))^{\beta})]} 
\end{equation} 

for a real number $A$, we get the following possible solutions

\begin{equation}
	u(t,x,y)=E_{\beta,1}(A(f(t))^{\beta})(x+y+2).
\end{equation} 

\quad

\section{Appendix: solving non-linear  partial differential equations using the method of invariant spaces}

 \quad

 We shall commence with an examination of the following vector space on $\mathbb{R}$
 
 $$ W=< f_{0}(x),...,f_{i}(x), ..., f_{n}(x)>=\{ \sum_{i=0}^{n} c_ {i}f_{i}(x) | c_{i} \in  \mathbb{R} \hspace{0.1cm} \forall i       \},$$  
 wherein $f_{i}(x)$ represent numerical functions with values in $\mathbb{R}$.
 In this context, we will say that $W$ is invariant under the action of an appropriate non-linear partial differential operator $L$ if 
 
 \begin{equation} 	
 \forall \hspace{0.1cm} v=\sum_{i=0}^{n} c_{i}f_{i}(x) \in \hspace{0.1cm} W \rightarrow L(v)=\sum_{i=0}^{n} \hat{L}_{i}(c_{0},...,c_{j},...,c_{n})f_{i}(x) \in W .\end{equation} 

These  vector spaces $W$ invariant under appropriate differential operators $L$ can be useful for the identification of potential solutions to differential equations of the type
 \begin{equation}
 	\mathcal{O}_{t}(u)=L(u),
 \end{equation}
 with $u=u(x,t)$ and $\mathcal{O}_{t}$ is a linear differential operator with respect to the variable $t$.
Indeed, the generators of W can be used to construct the following function $$ u(x,t)=\sum_{i=0}^{n}\tilde{u}_{i}(t)f_{i}(x),$$

which is a possible solution of (5.1) as soon as we find the coefficients $\tilde{u}_{i}(t)$, which are solutions of the system of ordinary differential equations with respect to the operator $\mathcal{O}_{t}$ that we are now going to determine.

$$ \mathcal{O}_{t}(\sum_{i=0}^{n}\tilde{u}_{i}(t) f_{i}(x))=L(\sum_{i=0}^{n}\tilde{u}_{i}(t)f_{i}(x))$$

$$\sum_{i=0}^{n}\mathcal{O}_{t}(\tilde{u}_{i}(t)) f_{i}(x)=\sum_{i=0}^{n} \hat{L}_{i}(\tilde{u}_{0},...,\tilde{u}_{j},...,\tilde{u}_{n})f_{i}(x)     $$

\begin{equation}
\sum_{i=0}^{n}(\mathcal{O}_{t}(\tilde{u}_{i}(t))-\hat{L}_{i}(\tilde{u}_{0},...,\tilde{u}_{j},...,\tilde{u}_{n})    )f_{i}(x)=0.
\end{equation}

Then (5.3) is satisfied if the $\tilde{u}_{i}(t)$ are solutions of the following system

\begin{equation}
	\begin{cases}
		\mathcal{O}_{t}(\tilde{u}_{0}(t))=\hat{L}_{0}(\tilde{u}_{0},...,\tilde{u}_{j},...,\tilde{u}_{n}) \\
		.\\
		.\\
		.\\
			\mathcal{O}_{t}(\tilde{u}_{i}(t))=\hat{L}_{i}(\tilde{u}_{0},...,\tilde{u}_{j},...,\tilde{u}_{n})\\
			.\\
			.\\
			.\\
				\mathcal{O}_{t}(\tilde{u}_{n}(t))=\hat{L}_{n}(\tilde{u}_{0},...,\tilde{u}_{j},...,\tilde{u}_{n})\\
	\end{cases}
\end{equation}  

The advantage of this method is evident in its ability to reduce the resolution of a non-linear partial derivative equation (5.1) into an ordinary system of differential equations (5.4).

\end{document}